\documentclass[10pt]{article}
\usepackage[utf8]{inputenc}
\usepackage{amsmath}
\usepackage{amssymb}
\usepackage{amsfonts}
\usepackage{enumitem}
\usepackage{amsthm}
\usepackage{geometry}
\usepackage{parskip}
\usepackage{graphicx}
\newcommand*\samethanks[1][\value{footnote}]{\footnotemark[#1]}

\title{Optimal control on graphs: existence, uniqueness, and long-term behavior\thanks{The authors would like to thank Guillaume Carlier (Université Paris Dauphine), Jean-Michel Lasry (Institut Louis Bachelier), and Jean-Michel Roquejoffre (Université Paul Sabatier) for the discussions they had on the subject.}}
\author{Olivier Guéant\thanks{Université Paris 1 Panthéon-Sorbonne. Centre d'Economie de la Sorbonne. 106, Boulevard de l'Hôpital, 75013 Paris.}, Iuliia Manziuk\samethanks[2]}
\date{}

\newtheorem{theorem}{Theorem}
\newtheorem{proposition}{Proposition}
\newtheorem{corollary}{Corollary}
\newtheorem{lemma}{Lemma}
\newtheorem{remark}{Remark}

\makeatletter
\def\blfootnote{\xdef\@thefnmark{}\@footnotetext}
\makeatother

\begin{document}

\maketitle

\abstract{The literature on continuous-time stochastic optimal control seldom deals with the case of discrete state spaces. In this paper, we provide a general framework for the optimal control of continuous-time Markov chains on finite graphs. In particular, we provide results on the long-term behavior of value functions and optimal controls, along with results on the associated ergodic Hamilton-Jacobi equation.}

\vspace{5mm}

\vspace{5mm}
\setlength{\parindent}{0em}

\blfootnote{Corresponding author: Pr. Olivier Guéant -- olivier.gueant@univ-paris1.fr}

\section{Introduction}

Optimal control is the field of mathematics dealing with the problem of the optimal way to control a dynamical system according to a given  optimality criterion. Since the 1950s and the seminal works of Bellman and Pontryagin, the number of successful applications have been so vast, and in so many domains, that optimal control theory can be regarded as one of the major contributions of applied mathematicians in the second half of the 20th century.\\

In spite of their widespread use, it is noteworthy that the theory of optimal control and that of stochastic optimal control (see for instance \cite{bertsekas2005dynamic}) have mainly been developed either in continuous time with a continuous state space, with tools coming from variational calculus, Euler-Lagrange equations, Hamilton-Jacobi(-Bellman) equations, the notion of viscosity solutions, etc. (see \cite{fleming2006controlled} for instance), or in discrete time, both on discrete and continuous state spaces, with contributions coming from both mathematics and computer science / machine learning (see the recent advances in reinforcement learning -- \cite{sutton2018reinforcement}).\\

Stochastic optimal control of continuous-time Markov chains on discrete state spaces is rarely tackled in the literature. It is part of the larger literature on the optimal control of point processes which has always been marginal (see \cite{bremaud1981point}) in spite of applications, for instance in finance -- see the literature on market making \cite{cartea2015algorithmic,gueant2016financial} which motivated our study.\\

In this short and modest paper, we aim at filling the gap by proposing a framework for the optimal control of continuous-time Markov chains on finite graphs. Using the classical mathematical techniques associated with Hamilton-Jacobi equations, we show the well-posedness of the differential equations characterizing the value functions and the existence of optimal controls. These results are elementary -- they do not need viscosity solutions -- and have already been derived in a similar manner for the more general case of mean field games on graphs (see \cite{gueant2015existence}). In this framework, we derive however a result that is absent from the literature in the case of the control of continuous-time Markov chains on discrete state spaces:  that of the long-term behavior of the value functions and the optimal controls, i.e. their behavior when the time horizon goes to infinity.\\

In the case of the optimal control of continuous-time Markov chains on connected finite graphs, under mild assumptions on the Hamiltonian functions that basically prevent the creation of several connected components in the graph, asymptotic results can in fact be obtained using simple tools: (i)~the existence of an ergodic constant is proved following the classical arguments of Lions, Papanicolaou and Varadhan (see \cite{lpv}), and (ii) the convergence of the ``de-drifted'' value function toward a solution of an ergodic Hamilton-Jacobi equation is proved -- following a discussion with Jean-Michel Roquejoffre -- using comparison principles and compactness results, that is, without relying on the use of KAM theory for Hamilton-Jacobi equations of order 1 as is the case in the classical results of Fathi in \cite{fathi}, nor on very specific assumptions regarding the Hamiltonian function as was the case in the initial results of Namah and Roquejoffre (see \cite{namah}) -- see also the paper of Barles and Souganidis \cite{bs} for general results of convergence in the case of a continuous state space. Moreover, we obtain the uniqueness (up to constants) of the solution to that ergodic Hamilton-Jacobi equation; a result that is not always true in the case of Hamilton-Jacobi equations of order 1.\\

In Section 2 we introduce the notations and describe both the framework and the initial finite-horizon control problem. In Section 3 we derive existence and uniqueness results for the solution of the Hamilton-Jacobi equation associated with the finite-horizon control problem and derive the optimal controls. In Section 4, we consider the infinite-horizon control problem with a positive discount rate and study the convergence of the stationary problem when the discount rate tends to~$0$. In Section 5, we use the results obtained in the stationary case to derive our main result: the asymptotic behavior of both the value functions and the optimal controls in the finite-horizon control problem when there is no discount.\\

\section{Notation and problem description}
\label{notation}
Let $T \in \mathbb{R}^*_+$. Let $\left(\Omega,\left(\mathcal{F}_{t}\right)_{t\in [0,T]},\mathbb{P}\right)$
be a filtered probability space, with $\left(\mathcal{F}_{t}\right)_{t\in [0,T]}$
satisfying the usual conditions. We assume that all stochastic processes introduced in this paper are defined on~$\Omega$ and adapted to the filtration $\left(\mathcal{F}_{t}\right)_{t\in [0,T]}$.\\

We consider a connected directed graph $\mathcal{G}$. The set of nodes are denoted by $\mathcal{I} = \{1, \ldots, N\}$. For each node $i \in \mathcal{I}$, we introduce $\mathcal{V}(i) \subset \mathcal{I} \setminus \{i\}$ the neighborhood of the node $i$, \emph{i.e.} the set of nodes $j$ for which a directed edge exists from $i$ to $j$. At any time $t \in [0,T]$, instantaneous transition probabilities are described by a collection of feedback control functions ${(\lambda_t(i, \cdot))}_{i \in \mathcal{I}}$ where $\lambda_t(i, \cdot): \mathcal{V}(i)\rightarrow \mathbb{R}_+$. We assume that the controls are in the admissible set $\mathcal{A}^T_0$ where, for $t\in [0,T]$,
\begin{align*}
    \mathcal{A}^T_t\!\! =\! \{&(\lambda_s(i, j))_{s \in [t, T], i \in \mathcal{I}, j \in \mathcal{V}(i)} \text{ non-negative, deterministic}| \forall i \in \mathcal{I}, \forall j \in \mathcal{V}(i), s \mapsto \lambda_s(i, j) \in L^{1}(t, T)\}.
\end{align*}

We consider an agent evolving on the graph $\mathcal{G}$. This agent can pay a cost to choose the values of the transition probabilities. We assume that the instantaneous cost of the agent located at node $i$ is described by a function $L(i, \cdot): \left(\lambda_{ij}\right)_{j \in \mathcal{V}(i)} \in {\mathbb{R}}^{|\mathcal{V}(i)|}_+ \mapsto L\left(i, \left(\lambda_{ij}\right)_{j \in \mathcal{V}(i)}\right) \in \mathbb{R} \cup \{+\infty\}$, where $|\mathcal{V}(i)|$ stands for the cardinality of the set $\mathcal{V}(i)$. The assumptions made on the functions $(L(i, \cdot))_{i \in \mathcal{I}}$ are the following:\footnote{Of course these functions can represent rewards if the value of the costs is negative.} \\

\begin{enumerate}[label=(A\arabic*)]
    \item \label{finiteAssump} Non-degeneracy: $\forall i \in \mathcal{I}, \exists \left(\lambda_{ij}\right)_{j \in \mathcal{V}(i)} \in {\mathbb{R}}^{*|\mathcal{V}(i)|}_+, L\left(i, \left(\lambda_{ij}\right)_{j \in \mathcal{V}(i)}\right) <+\infty$;
    \item \label{lscAssump} Lower semi-continuity: $\forall i \in \mathcal{I}$, $L(i, \cdot)$ is lower semi-continuous;
    \item \label{boundAssump} Boundedness from below: $\exists \underline{C} \in \mathbb{R}$, $\forall i \in \mathcal{I}$, $\forall \left(\lambda_{ij}\right)_{j \in \mathcal{V}(i)} \in {\mathbb{R}}^{|\mathcal{V}(i)|}_+$, $L\left(i, \left(\lambda_{ij}\right)_{j \in \mathcal{V}(i)}\right) \ge \underline{C}$;
    \item Asymptotic super-linearity: \label{asympAssump}\begin{align}
        & \forall i \in \mathcal{I}, \lim_{\left\|\left(\lambda_{ij}\right)_{j \in \mathcal{V}(i)}\right\|_{\infty} \rightarrow +\infty} \frac{L\left(i, \left(\lambda_{ij}\right)_{j \in \mathcal{V}(i)}\right)}{\left\|\left(\lambda_{ij}\right)_{j \in \mathcal{V}(i)}\right\|_{\infty}} = +\infty.
    \end{align}
\end{enumerate}

At time $T$, we consider a terminal payoff for the agent. This payoff depends on his position on the graph and is modelled by a function $g: \mathcal{I} \rightarrow \mathbb{R}$.\\

Let us denote by $(X_s^{t, i, \lambda})_{s \in [t, T]}$ the continuous-time Markov chain on $\mathcal{G}$ starting from node $i$ at time~$t$, with instantaneous transition probabilities given by $\lambda \in \mathcal{A}^T_t$.\\

Starting from a given node $i$ at time $0$, the control problem we consider is the following:\\
\begin{align}
\label{controlpbm}
    & \sup_{\lambda \in \mathcal{A}^T_0} \mathbb{E}\left[ -\int_0^T e^{-rt} L\left(X^{0,i,\lambda}_t, \left(\lambda_t\left(X^{0,i,\lambda}_t, j\right)\right)_{j \in \mathcal{V}\left(X^{0,i,\lambda}_t\right)}\right)dt + e^{-rT} g\left(X^{0,i,\lambda}_T\right)\right],
\end{align}
where $r \ge 0$ is a discount rate.  \\

For each node $i \in \mathcal{I}$, the value function of the agent is defined as
\begin{align*}
    & u^{T,r}_i(t) = \sup_{\lambda \in \mathcal{A}^T_t} \mathbb{E}\left[ -\int_t^T e^{-r(s-t)} L\left(X_s^{t, i, \lambda}, \left(\lambda_s\left(X_s^{t, i, \lambda}, j\right)\right)_{j \in \mathcal{V}\left(X^{t,i,\lambda}_s\right)}\right)ds + e^{-r(T-t)} g\left(X_T^{t, i, \lambda}\right)\right].
\end{align*}

The Hamilton-Jacobi equation associated with the above optimal control problem is
\begin{eqnarray*}
    \forall i \in \mathcal{I},\quad 0 &=& \frac{d}{dt}{V^{T,r}_i}(t) - r V^{T,r}_i(t)\\
     &&+\sup_{\left(\lambda_{ij}\right)_{j \in \mathcal{V}(i)} \in \mathbb{R}^{|\mathcal{V}(i)|}_+} \left( \left(\sum_{j \in \mathcal{V}(i)}\lambda_{ij}\left(V^{T,r}_j(t) - V^{T,r}_i(t)\right) \right) - L\left(i, \left(\lambda_{ij}\right)_{j \in \mathcal{V}(i)}\right)\right),
\end{eqnarray*}
with terminal condition
\begin{align}
\label{terminal condition}
    V^{T,r}_i(T) = g(i), \quad \forall i \in \mathcal{I}.
\end{align}

Let us define the Hamiltonian functions associated with the cost functions $(L(i, \cdot))_{i \in \mathcal{I}}$:
\begin{align*}
    & \forall i \in \mathcal{I}, H(i,\cdot): p \in \mathbb{R}^{|\mathcal{V}(i)|} \mapsto H(i, p) = \sup_{\left(\lambda_{ij}\right)_{j \in \mathcal{V}(i)} \in {\mathbb{R}}^{|\mathcal{V}(i)|}_+}  \left(\left(\sum_{j \in \mathcal{V}(i)} \lambda_{ij} p_j\right) - L\left(i, \left(\lambda_{ij}\right)_{j \in \mathcal{V}(i)}\right)\right).
\end{align*}
Then the above Hamilton-Jacobi equation can be reformulated as
\begin{align}
    \label{mainEQH}
    & \frac{d}{dt}{V^{T,r}_i}(t) - r V^{T,r}_i(t) + H\left(i, \left(V^{T,r}_j(t) - V^{T,r}_i(t)\right)_{j \in \mathcal{V}(i)}\right) = 0, \quad \forall (i,t) \in \mathcal{I}\times[0,T].
\end{align}

Our goal in the next section is to prove that there exists a unique strong solution to Eq. \eqref{mainEQH} with terminal condition \eqref{terminal condition}.

\section{Existence and uniqueness of the solution to the Hamilton-Jacobi equation}
\label{fh}

In order to prove existence and uniqueness for the Hamilton-Jacobi equation \eqref{mainEQH} with terminal condition \eqref{terminal condition}, we start with a proposition on the Hamiltonian functions.

\begin{proposition}
\label{H}
$\forall i \in \mathcal{I}$, the function $H(i, \cdot)$ is well defined (i.e. finite) and verifies the following properties:\\

\begin{enumerate}[label=(P\arabic*)]
\item \label{reached} $\forall p=(p_j)_{j \in \mathcal{V}(i)} \in \mathbb{R}^{|\mathcal{V}(i)|}, \exists \left(\lambda^*_{ij}\right)_{j \in \mathcal{V}(i)} \in {\mathbb{R}}^{|\mathcal{V}(i)|}_+,  H(i,p) = \left(\sum_{j \in \mathcal{V}(i)} \lambda^*_{ij} p_j\right) - L\left(i, \left(\lambda^*_{ij}\right)_{j \in \mathcal{V}(i)}\right).$\\
\item \label{Hconvex}$H(i, \cdot)$ is convex on $\mathbb{R}^{|\mathcal{V}(i)|}$.\\
\item \label{Hnd}$H(i, \cdot)$ is non-decreasing with respect to each coordinate.\\
\end{enumerate}
\end{proposition}
\begin{proof}

Let us consider $p=\left(p_j\right)_{j \in \mathcal{V}(i)} \in \mathbb{R}^{|\mathcal{V}(i)|}$.\\

From assumption \ref{finiteAssump}, we can consider $\left(\tilde{\lambda}_{ij}\right)_{j \in \mathcal{V}(i)} \in {\mathbb{R}}^{|\mathcal{V}(i)|}_+$ such that $L\left(i, \left(\tilde \lambda_{ij}\right)_{j \in \mathcal{V}(i)}\right) < +\infty$.\\

We then use assumption \ref{asympAssump} on the function $L(i, \cdot)$ to derive the existence of a positive number~$A$ such that
$$
    \forall \left(\lambda_{ij}\right)_{j \in \mathcal{V}(i)} \in {\mathbb{R}}^{|\mathcal{V}(i)|}_+, \left\|\left(\lambda_{ij}\right)_{j \in \mathcal{V}(i)}\right\|_{\infty} \ge A \Rightarrow  L\left(i, \left(\lambda_{ij}\right)_{j \in \mathcal{V}(i)} \right) \ge \left(1 + \left\|p\right\|_\infty |\mathcal{V}(i)|\right)  \left\|\left(\lambda_{ij}\right)_{j \in \mathcal{V}(i)}\right\|_{\infty}.
$$

Let us define $C = \max\left(A, L\left(i, \left(\tilde \lambda_{ij}\right)_{j \in \mathcal{V}(i)}\right) - \left(\sum_{j \in \mathcal{V}(i)} \tilde \lambda_{ij} p_j\right), \left\|\left(\tilde \lambda_{ij}\right)_{j \in \mathcal{V}(i)}\right\|_\infty\right)$.\\

For $\left(\lambda_{ij}\right)_{j \in \mathcal{V}(i)} \in {\mathbb{R}}^{|\mathcal{V}(i)|}_+$, if $\left\|\left(\lambda_{ij}\right)_{j \in \mathcal{V}(i)}\right\|_{\infty} \ge C$, then
\begin{eqnarray*}
\left(\sum_{j \in \mathcal{V}(i)} \lambda_{ij} p_j\right) - L\left(i, \left(\lambda_{ij}\right)_{j \in \mathcal{V}(i)}\right) \le \left\|\left(\lambda_{ij}\right)_{j \in \mathcal{V}(i)}\right\|_{\infty} \left\|p\right\|_\infty  |\mathcal{V}(i)| - \left(1 + \left\|p\right\|_\infty |\mathcal{V}(i)|\right)  \left\|\left(\lambda_{ij}\right)_{j \in \mathcal{V}(i)}\right\|_{\infty}\\
\end{eqnarray*}
\begin{eqnarray*}
&\le& - \left\|\left(\lambda_{ij}\right)_{j \in \mathcal{V}(i)}\right\|_{\infty}\\
&\le& \left(\sum_{j \in \mathcal{V}(i)} \tilde \lambda_{ij} p_j\right) - L\left(i, \left(\tilde \lambda_{ij}\right)_{j \in \mathcal{V}(i)}\right).
\end{eqnarray*}
Therefore
$$\sup_{\left(\lambda_{ij}\right)_{j \in \mathcal{V}(i)} \in {\mathbb{R}}^{|\mathcal{V}(i)|}_+}  \left(\left(\sum_{j \in \mathcal{V}(i)} \lambda_{ij} p_j\right) - L\left(i, \left(\lambda_{ij}\right)_{j \in \mathcal{V}(i)}\right)\right)$$$$=\sup_{\left(\lambda_{ij}\right)_{j \in \mathcal{V}(i)} \in \mathcal{C}}  \left(\left(\sum_{j \in \mathcal{V}(i)} \lambda_{ij} p_j\right) - L\left(i, \left(\lambda_{ij}\right)_{j \in \mathcal{V}(i)}\right)\right),$$
where $\mathcal{C} = \left\{ \left(\lambda_{ij}\right)_{j \in \mathcal{V}(i)} \in {\mathbb{R}}^{|\mathcal{V}(i)|}_+, \left\|\left(\lambda_{ij}\right)_{j \in \mathcal{V}(i)}\right\|_\infty \le C\right\}$. By using assumption \ref{lscAssump}, we obtain that the supremum is reached on the compact set~$\mathcal{C}$.\\

Regarding the convexity of $H(i, \cdot)$, we simply need to write it as a Legendre-Fenchel transform (denoted hereafter by the sign $^\star$):
\begin{eqnarray*}
H(i,p) &=& \sup_{\left(\eta_{ij}\right)_{j \in \mathcal{V}(i)} \in \mathbb{R}^{|\mathcal{V}(i)|}} \left( \left(\sum_{j \in \mathcal{V}(i)} \eta_{ij} p_j\right)  - \left(L\left(i, \left(\eta_{ij}\right)_{j \in \mathcal{V}(i)}\right) + \chi\left((\eta_{ij})_{j \in \mathcal{V}(i)}\right)\right)\right)\\
 &=& (L(i, \cdot) + \chi(\cdot))^\star(p),
\end{eqnarray*}
where \begin{align*}
    & \chi: \eta \in \mathbb{R}^{|\mathcal{V}(i)|} \mapsto \begin{cases}
    0, \text{ if } \eta \in \mathbb{R}^{|\mathcal{V}(i)|}_+\\
    +\infty, \text{ otherwise}.\\
    \end{cases}
\end{align*}

Let us prove now that $H(i, \cdot)$ is non-decreasing with respect to each coordinate. Let us consider $p=(p_j)_{j \in \mathcal{V}(i)}$ and $p'=(p'_j)_{j \in \mathcal{V}(i)}$ such that $\forall j \in \mathcal{V}(i), p_j \ge p'_j$. Then we have
$$ \forall \left(\lambda_{ij}\right)_{j \in \mathcal{V}(i)} \in {\mathbb{R}}^{|\mathcal{V}(i)|}_+,  \left(\sum_{j \in \mathcal{V}(i)} \lambda_{ij} p_j\right) - L\left(i, \left(\lambda_{ij}\right)_{j \in \mathcal{V}(i)}\right) \ge \left(\sum_{j \in \mathcal{V}(i)} \lambda_{ij} p'_j\right) - L\left(i, \left(\lambda_{ij}\right)_{j \in \mathcal{V}(i)}\right).$$
By taking the supremum on both sides, we obtain $H(i, p) \ge H(i, p')$, hence the result.
\end{proof}

We now turn to a central result for existence and uniqueness: a comparison principle that applies to the Hamilton-Jacobi equation.

\begin{proposition}[Comparison principle]
\label{CompPrinc}
Let $t' \in (-\infty,T)$. Let $(v_i)_{i \in \mathcal{I}}$ and $(w_i)_{i \in \mathcal{I}}$ be two continuously differentiable functions on $[t',T]$ such that
\begin{align}
\label{subsolution}
    & \frac{d}{dt}v_i(t) - r v_i(t) + H\left(i, \left(v_j(t) - v_i(t)\right)_{j \in \mathcal{V}(i)}\right)  \ge 0,\quad \forall (i,t) \in \mathcal{I} \times [t', T],\\
\label{supersolution}
    & \frac{d}{dt}w_i(t) - r w_i(t) + H\left(i, \left(w_j(t) - w_i(t)\right)_{j \in \mathcal{V}(i)}\right)  \le 0,\quad  \forall (i,t) \in \mathcal{I} \times [t', T],
\end{align}
and $v_i(T) \le w_i(T), \forall i \in \mathcal{I}$.\\

Then $v_i(t) \le w_i(t)$, $ \forall (i,t) \in \mathcal{I} \times [t', T]$.\\
\end{proposition}
\begin{proof}
Let $\varepsilon > 0$. Let us define $z: (i,t) \in \mathcal{I} \times [t', T] \mapsto z_i(t) = e^{-rt} (v_i(t) - w_i(t) - \varepsilon (T - t))$.
We have
\begin{align*}
    \frac{d}{dt}z_i(t) &= -r e^{-rt} (v_i(t) - w_i(t) -  \varepsilon (T - t)) + e^{-rt}\left(\frac{d}{dt}v_i(t) - \frac{d}{dt}w_i(t) + \varepsilon\right) \\
    & = e^{-rt} \left(\left(\frac{d}{dt}v_i(t) - rv_i(t)\right) - \left(\frac{d}{dt}w_i(t) - rw_i(t)\right) +\varepsilon + r\varepsilon(T - t)\right) \\
    & \ge e^{-rt} \left(-H\left(i, \left(v_j(t) - v_i(t)\right)_{j \in \mathcal{V}(i)}\right)  + H\left(i, \left(w_j(t) - w_i(t)\right)_{j \in \mathcal{V}(i)}\right) + \varepsilon + r\varepsilon(T - t)\right).
\end{align*}
Let us choose $(i^*, t^*) \in \mathcal{I} \times [t', T]$ maximizing $z$.\\

If $t^* < T$, then $\frac{d}{dt}z_{i^*}\left(t^*\right) \le 0$. Therefore,
    \[
    H\left(i^*, \left((v_j\left(t^*\right) - v_{i^*}\left(t^*\right)\right)_{j \in \mathcal{V}(i^*)}\right) \ge H\left(i^*, \left((w_j\left(t^*\right) - w_{i^*}\left(t^*\right)\right)_{j \in \mathcal{V}(i^*)}\right) + \varepsilon + r\varepsilon(T - t^*).
    \]

    By definition of $(i^*, t^*)$, we know that $\forall j \in \mathcal{V}(i^*), v_{j}\left(t^*\right) - w_{j}\left(t^*\right) \le v_{i^*}\left(t^*\right) - w_{i^*}\left(t^*\right)$, and therefore $\forall j \in \mathcal{V}(i^*), v_{j}\left(t^*\right) - v_{i^*}\left(t^*\right) \le w_{j}\left(t^*\right) - w_{i^*}\left(t^*\right)$.\\

    From \ref{Hnd}, it follows that
    \begin{align*}
        H\left(i^*, \left(v_{j}\left(t^*\right) - v_{i^*}\left(t^*\right)\right)_{j \in \mathcal{V}(i^*)}\right) \le H\left(i^*, \left(w_{j}\left(t^*\right) - w_{i^*}\left(t^*\right)\right)_{j \in \mathcal{V}(i^*)}\right).
    \end{align*}
    This contradicts the above inequality. Therefore, $t^* = T$, and we have:
    \[ \forall (i,t) \in \mathcal{I} \times [t', T],\quad   z_i(t) \le z_{i^*}(T) = e^{-rT}(v_{i^*}(T) - w_{i^*}(T)) \le 0.\]
Therefore, $\forall (i,t) \in \mathcal{I} \times [t', T], \quad  v_i(t) \le w_i(t) + \varepsilon(T - t).$
We conclude by sending $\varepsilon$ to~$0$.
\end{proof}

We are now ready to prove existence and uniqueness.

\begin{theorem}[Global existence and uniqueness]
\label{global}There exists a unique solution $\left(V^{T,r}_i\right)_{i \in \mathcal I}$  to Eq. \eqref{mainEQH} on $(-\infty,T]$ with terminal condition \eqref{terminal condition}.
\end{theorem}
\begin{proof}
$\forall i \in \mathcal{I}$, the function $H(i, \cdot)$ is locally Lipschitz because of \ref{Hconvex}. Therefore we can apply Cauchy–Lipschitz local existence and uniqueness theorem to Eq. \eqref{mainEQH} with terminal condition~\eqref{terminal condition}. Therefore there exists a maximal solution $\left(V^{T,r}_i\right)_{i \in \mathcal I}$ defined over $(\tau^*, T]$, where $\tau^* \in [-\infty, T)$.\\

Let us prove by contradiction that $\tau^* = -\infty$.\\

First of all, let us provide \emph{a priori} bounds to the functions $\left(V^{T,r}_i\right)_{i \in \mathcal I}$. For that purpose, let us define for $C \in \mathbb R$ to be chosen, the function $$ v^C : (i,t) \in \mathcal{I} \times (\tau^*, T] \mapsto v^C_i(t) = e^{-r(T - t)}\left(g(i) + C(T - t)\right).$$
We have $\forall i \in \mathcal{I}, v^C_i(T) = g(i)$ and
\begin{equation*}
\begin{aligned}
    & \frac{d}{dt}{v^C_i}(t) - r v^C_i(t) + H\left(i, \left(v^C_j(t) - v^C_i(t)\right)_{j \in \mathcal{V}(i)}\right) \\
     = & - Ce^{-r(T - t)} + H\left(i, e^{-r(T - t)}\left(g(j)-g(i)\right)_{j \in \mathcal{V}(i)}\right), \quad \forall (i,t) \in \mathcal{I} \times (\tau^*, T].
\end{aligned}
\end{equation*}
If $\tau^*$ is finite, the function $(i,t) \in \mathcal{I} \times (\tau^*, T] \mapsto e^{r(T - t)} H\left(i, e^{-r(T - t)}(g(j)-g(i))_{j \in \mathcal{V}(i)}\right)$ is bounded, hence the existence of two constants $C_1$ and $C_2$ such that $\forall (i,t) \in \mathcal{I} \times (\tau^*, T]$,
\begin{align*}
    &- C_1 e^{-r(T - t)} + H\left(i, e^{-r(T - t)}(g(j)-g(i))_{j \in \mathcal{V}(i)}\right) \ge 0, \quad \text{and} \\
    &- C_2 e^{-r(T - t)} + H\left(i, e^{-r(T - t)}(g(j)-g(i))_{j \in \mathcal{V}(i)}\right) \le 0.
\end{align*}
We can therefore apply the above comparison principle (Proposition \ref{CompPrinc}) to $v^{C_1}$ and $V^{T,r}$, and then to $V^{T,r}$ and $v^{C_2}$ over any interval $[t',T] \subset (\tau^*,T]$ to obtain:
$$ \forall (i,t) \in \mathcal{I} \times [t', T], \quad v^{C_1}_i(t) \le V^{T,r}_i(t) \le v^{C_2}_i(t).$$
Then, by sending $t'$ to $\tau^*$ we obtain that
$$ \forall (i,t) \in \mathcal{I} \times (\tau^*, T], \quad v^{C_1}_i(t) \le V^{T,r}_i(t) \le v^{C_2}_i(t).$$
In particular, $\tau^*$ finite implies that the functions $\left(V^{T,r}_i\right)_{i \in \mathcal I}$ are bounded.\\

Now, let us define for $h \in \mathbb R$ to be chosen, the function $$ w^h : (i,t) \in \mathcal{I} \times (\tau^*, T] \mapsto w^h_i(t) = e^{r(T - t)} V^{T,r}_i(t) + h(T-t).$$
We have
\begin{align*}
    \frac{d}{dt}{w^h_i}(t) & = -re^{r(T - t)}V^{T,r}_i(t) + e^{r(T - t)}\frac{d}{dt}{V^{T,r}_i}(t) - h \\
    & = e^{r(T - t)}\left(\frac{d}{dt}{V^{T,r}_i}(t) - rV^{T,r}_i(t)\right) - h \\
    & = - e^{r(T - t)} H\left(i, \left(V^{T,r}_j(t) - V^{T,r}_i(t)\right)_{j \in \mathcal{V}(i)}\right) - h.
\end{align*}
If $\tau^*$ is finite, using the boundedness of $\left(V^{T,r}_i\right)_{i \in \mathcal I}$, $\exists h \in \mathbb{R}, \forall (i,t) \in \mathcal{I} \times (\tau^*, T], \frac{d}{dt}{w^h_i}(t) \le 0$. Therefore $\lim_{t \rightarrow \tau^*, t > \tau^*} w_i(t)$ exists $\forall i \in \mathcal{I}$, so $\lim_{t \rightarrow \tau^*, t > \tau^*} V^{T,r}_i(t)$ exists $\forall i \in \mathcal{I}$, and it is finite as the functions $\left(V^{T,r}_i\right)_{i \in \mathcal I}$ are bounded. Thus we obtain a contradiction with the maximality of the solution $\left(V^{T,r}_i\right)_{i \in \mathcal I}$.\\

We conclude that $\tau^* = -\infty$, and that there exists a unique solution to Eq. \eqref{mainEQH} on $(-\infty,T]$ with terminal condition \eqref{terminal condition}.\\
\end{proof}

\begin{remark}
In the proof of the above results, the convexity of the Hamiltonian functions $(H(i,\cdot))_{i \in \mathcal I}$ does not play any role. The results indeed hold as soon as the Hamiltonian functions are locally Lipschitz and non-decreasing with respect to each coordinate.\\
\end{remark}

By using a standard verification argument, we obtain the solution to our initial control problem. This is the purpose of the next theorem.\\

\begin{theorem}
\label{theocontrol}
We have:\\

\begin{itemize}
\item $\forall (i,t) \in \mathcal I \times [0,T], u^{T,r}_i(t) = V^{T,r}_i(t)$.\\
\item The optimal controls for Problem \eqref{controlpbm} are given by any feedback control function verifying for all $i \in \mathcal I$, for all $j \in \mathcal V(i)$, and for all $t \in [0,T]$, $$\lambda^*_t(i,j)\! \in\! \underset{\left(\lambda_{ij}\right)_{j \in \mathcal{V}(i)} \in {\mathbb{R}}^{|\mathcal{V}(i)|}_+ }{\textrm{argmax}} \left( \left(\sum_{j \in \mathcal{V}(i)} \lambda_{ij} \left(u^{T,r}_j(t) - u^{T,r}_i(t)\right)\right) - L\left(i, \left(\lambda_{ij}\right)_{j \in \mathcal{V}(i)}\right)\right).$$
\end{itemize}
\end{theorem}

\begin{remark}
    The $\textrm{argmax}$ in Theorem \ref{theocontrol} is a singleton if the Hamiltonian functions $(H(i, \cdot))_i$ are differentiable (which is guaranteed if $(L(i, \cdot))_i$ are convex functions that are strictly convex on their respective domain).\\
\end{remark}

\section{Infinite-horizon problem: from the stationary to the ergodic case}
\label{stationary}
In this section we consider $r > 0$.\\

Our first goal is to obtain the convergence when $T \to +\infty$ of the above control problem towards the infinite-horizon / stationary control problem. Then, our second and main goal is to state what happens when $r \to 0$.  Similar results can be found for instance in \cite{bcd} in the case of a continuous state space.\\

Let us first state the convergence result corresponding to the limit case $T \to +\infty$.\\

\begin{proposition}We have
$$\exists (u_i^r)_{i \in \mathcal{I}} \in \mathbb R^N, \forall (i,t) \in \mathcal I \times \mathbb{R}_+, \lim_{T \to +\infty} u^{T,r}_i(t) = u_i^r.$$
Furthermore, $(u_i^r)_{i \in \mathcal{I}}$ satisfies the following Bellman equation:
\begin{equation}\label{HJstat}
  - r u^r_i + H\left(i, \left(u_j^{r} - u_i^{r}\right)_{j \in \mathcal{V}(i)}\right) = 0,\quad  \forall i \in \mathcal I.
\end{equation}
\end{proposition}

\begin{proof}
Let us define $$u_i^r = \sup_{\lambda \in \mathcal{A}^{\infty}_0} \mathbb{E}\left[ -\int_0^{+\infty} e^{-rt} L\left(X^{0,i,\lambda}_t, \left(\lambda_t\left(X^{0,i,\lambda}_t, j\right)\right)_{j \in \mathcal{V}\left(X^{0,i,\lambda}_t\right)}\right)dt\right], \quad \forall i  \in \mathcal I,$$
where \begin{eqnarray*}
    \mathcal{A}^\infty_t &=& \left\{\left(\lambda_s(i, j)\right)_{s \in [t, +\infty), i \in \mathcal{I}, j \in \mathcal{V}(i)} \text{ non-negative, deterministic}|\right.\\
                        && \left. \vphantom{\left(\lambda_s(i, j)\right)_{s \in [t, +\infty), i \in \mathcal{I}, j \in \mathcal{V}(i)}}  \forall i \in \mathcal{I}, \forall j \in \mathcal{V}(i), s \mapsto \lambda_s(i, j) \in L_{\textrm{loc}}^{1}(t, +\infty)\right\}.
\end{eqnarray*}

Let us consider an optimal control $\lambda^* \in \mathcal{A}_0^T$ over $[0,T]$ as in Theorem \ref{theocontrol}. We define a control $\lambda \in \mathcal{A}^\infty_0$ by $\lambda_t = \lambda^*_t$ for $t \in [0,T]$ and $(\lambda_t(i,j))_{i \in \mathcal I, j \in \mathcal{V}(i)} = \left(\tilde{\lambda}_{ij}\right)_{ i \in \mathcal I,  j \in \mathcal{V}(i)}$ for $t>T$, where $\tilde{\lambda}$ is as in \ref{finiteAssump}. We have for all $i \in \mathcal I$,
\begin{eqnarray*}
u^r_i &\ge& \mathbb{E}\left[ -\int_0^\infty e^{-rt} L\left(X^{0,i,\lambda}_t, \left(\lambda_t\left(X^{0,i,\lambda}_t, j\right)\right)_{j \in \mathcal{V}\left(X^{0,i,\lambda}_t\right)}\right)dt\right]\\
&\ge& \mathbb{E}\left[ -\int_0^T e^{-rt} L\left(X^{0,i,\lambda^*}_t, \left(\lambda^*_t\left(X^{0,i,\lambda^*}_t, j\right)\right)_{j \in \mathcal{V}\left(X^{0,i,\lambda^*}_t\right)}\right)dt\right]\\
&&{}+ \mathbb{E}\left[ -\int_T^\infty e^{-rt} L\left(X^{T,X^{0,i,\lambda^*}_T,\lambda}_t, \left(\lambda_t\left(X^{T,X^{0,i,\lambda^*}_T,\lambda}_t, j\right)\right)_{j \in \mathcal{V}\left(X^{T,X^{0,i,\lambda^*}_T,\lambda}_t\right)}\right)dt\right]\\
&\ge& u_i^{T,r}(0) - e^{-rT} g\left(X^{0,i,\lambda^*}_T\right)\\
&&\!\!\!{}+ e^{-rT}\mathbb{E}\left[ -\int_T^\infty e^{-r(t-T)} L\left(X^{T,X^{0,i,\lambda^*}_T,\lambda}_t, \left(\lambda_t\left(X^{T,X^{0,i,\lambda^*}_T,\lambda}_t, j\right)\right)_{j \in \mathcal{V}\left(X^{T,X^{0,i,\lambda^*}_T,\lambda}_t\right)}\right)dt\right].\\
\end{eqnarray*}
Given the definition of $\lambda$ over $(T,+\infty)$ there exists $C$ such that $$\forall t > T, \quad L\left(X^{T,X^{0,i,\lambda^*}_T,\lambda}_t, \left(\lambda_t\left(X^{T,X^{0,i,\lambda^*}_T,\lambda}_t, j\right)\right)_{j \in \mathcal{V}\left(X^{T,X^{0,i,\lambda^*}_T,\lambda}_t\right)}\right) \le C.$$ Therefore,
$$ u^r_i \ge u_i^{T,r}(0) - e^{-rT} g\left(X^{0,i,\lambda^*}_T\right) - e^{-rT} \frac{C}r,$$
hence $\limsup_{T \to +\infty} u_i^{T,r}(0) \le u^r_i$.\\

Let us consider $\varepsilon>0$. Let us consider $\lambda^\varepsilon \in \mathcal{A}^\infty_0$ such that
$$u^r_i - \varepsilon \le \mathbb{E}\left[ -\int_0^\infty e^{-rt} L\left(X^{0,i,\lambda^\varepsilon}_t, \left(\lambda^\varepsilon_t\left(X^{0,i,\lambda^\varepsilon}_t, j\right)\right)_{j \in \mathcal{V}\left(X^{0,i,\lambda^\varepsilon}_t\right)}\right)dt\right].$$
We have
\begin{eqnarray*}
u^r_i - \varepsilon &\le& \mathbb{E}\left[ -\int_0^T e^{-rt} L\left(X^{0,i,\lambda^\varepsilon}_t, \left(\lambda^\varepsilon_t\left(X^{0,i,\lambda^\varepsilon}_t, j\right)\right)_{j \in \mathcal{V}\left(X^{0,i,\lambda^\varepsilon}_t\right)}\right)dt\right]\\
&&{}+ \mathbb{E}\left[ -\int_T^\infty e^{-rt} L\left(X^{T, X^{0,i,\lambda^\varepsilon}_T ,\lambda^\varepsilon}_t, \left(\lambda^\varepsilon_t\left(X^{T, X^{0,i,\lambda^\varepsilon}_T,\lambda^\varepsilon}_t, j\right)\right)_{j \in \mathcal{V}\left(X^{T, X^{0,i,\lambda^\varepsilon}_T,\lambda^\varepsilon}_t\right)}\right)dt\right] \\
&\le& u_i^{T,r}(0) -  e^{-rT} g\left(X^{0,i,\lambda^\varepsilon}_T\right) + e^{-rT} \frac{\underline{C}}{r},
\end{eqnarray*}
where $\underline{C}$ is defined in \ref{boundAssump}.\\

Therefore $\liminf_{T \to +\infty} u_i^{T,r}(0) \ge u^r_i - \varepsilon.$\\

By sending $\varepsilon$ to $0$, we obtain $\forall i \in \mathcal{I}, \lim_{T \to +\infty} u_i^{T,r}(0) = u^r_i$.\\

Let us notice now that $$\forall i \in \mathcal I, \forall s,t \in \mathbb{R}_+, \forall T > t, u^{T+s,r}_i(t) = u^{T+s-t,r}_i(0) = V^{T,r}_i(t-s).$$

Therefore $\forall (i,t) \in \mathcal{I}\times\mathbb R_+, \lim_{T \to +\infty} u_i^{T,r}(t) = u^r_i = \lim_{s \to -\infty} V^{T,r}_i(s)$.\\

Using Eq. \eqref{mainEQH}, we see that if $\left(V^{T,r}_i\right)_{i \in \mathcal I}$ has a finite limit in $-\infty$, then so does $\frac{d}{dt}\left({V^{T,r}_i}\right)_{i \in \mathcal I}$. But, then, necessarily, the latter limit is equal to nought, as otherwise the former could not be finite. By passing to the limit in Eq. \eqref{mainEQH}, we obtain
\begin{equation*}
  - r u^r_i + H\left(i, \left(u_j^{r} - u_i^{r}\right)_{j \in \mathcal{V}(i)}\right) = 0,\quad  \forall i \in \mathcal I.
\end{equation*}
\end{proof}

Using a standard verification argument, we obtain a simpler characterization of the limit:\\

\begin{proposition}Let $\mathcal{A}$ be the set of non-negative families $(\lambda(i, j))_{i \in \mathcal{I}, j \in \mathcal{V}(i)}$. Then
$$u_i^r = \sup_{\lambda \in \mathcal{A}} \mathbb{E}\left[ -\int_0^{+\infty} e^{-rt} L\left(X^{0,i,\lambda}_t, \left(\lambda\left(X^{0,i,\lambda}_t, j\right)\right)_{j \in \mathcal{V}\left(X^{0,i,\lambda}_t\right)}\right)dt\right], \quad \forall i  \in \mathcal I.$$
\end{proposition}

We now come to the study of the limit case $r \to 0$, which, as we shall see, corresponds to the convergence towards the ergodic problem. We start with the following lemma:\\

\begin{lemma}
\label{lemmaboundedr}
We have:
\begin{enumerate}[label=(\roman*)]
    \item $\forall i \in \mathcal{I}$, $r \in \mathbb{R}_+^* \mapsto ru_i^r$ is bounded;
    \item $\forall i \in \mathcal{I}$, $\forall j \in \mathcal{V}(i)$, $r \in \mathbb{R}_+^* \mapsto u_j^r - u_i^r$ is bounded.
\end{enumerate}
\end{lemma}

\begin{proof}
\begin{enumerate}[label=(\roman*)]
\item Let us choose $(\lambda(i,j))_{i \in \mathcal{I}, j \in \mathcal{V}(i)} \in \mathcal{A}$ as in assumption \ref{finiteAssump}. By definition of $u^r_i$ we have
\begin{eqnarray*}
    u_i^r &\ge& \mathbb{E}\left[ -\int_0^{+\infty} e^{-rt} L\left(X^{0,i,\lambda}_t, \left(\lambda\left(X^{0,i,\lambda}_t, j\right)\right)_{j \in \mathcal{V}\left(X^{0,i,\lambda}_t\right)}\right)dt\right]\\
    & \ge& \int_0^{+\infty} e^{-rt} \inf_k -L\left(k, \left(\lambda(k,j)\right)_{j \in \mathcal{V}(k)}\right)dt\\
    &\ge& \frac{1}{r} \inf_k -L\left(k, \left(\lambda(k,j)\right)_{j \in \mathcal{V}(k)}\right).
\end{eqnarray*}
From the boundedness assumption of the functions $(L(i,\cdot))_{i \in \mathcal{I}}$ (see
\ref{boundAssump}), we also have for all $(\lambda(i,j))_{i \in \mathcal{I}, j \in \mathcal{V}(i)} \in \mathcal{A}$ that
$$\mathbb{E}\left[ -\int_0^{+\infty} e^{-rt} L\left(X^{0,i,\lambda}_t, \left(\lambda\left(X^{0,i,\lambda}_t, j\right)\right)_{j \in \mathcal{V}\left(X^{0,i,\lambda}_t\right)}\right)dt\right]  \le  - \underline{C} \int_0^{+\infty} e^{-rt}  dt  = -\frac{\underline{C}}{r}.$$
Therefore, $u_i^r\le -\frac{\underline{C}}{r}$.\\

We conclude that $r \mapsto ru_i^r$ is bounded.\\

\item Let us consider a family of positive intensities $\left(\lambda(i,j)\right)_{i \in \mathcal{I}, j \in \mathcal{V}(i)} \in \mathcal{A}$ as in assumption \ref{finiteAssump}. Because $\mathcal G$  is connected, the positiveness of the above intensities implies that for all $(i,j) \in \mathcal{I}^2$ the stopping time defined by $\tau^{ij} = \inf\left\{t\Big|X^{0,i,\lambda}_t = j \right\}$ verifies $\mathbb{E}\left[\tau^{ij}\right] < +\infty$.\\

Now, $\forall (i,j) \in \mathcal{I}^2$, we have
\begin{eqnarray*}
     &&u_i^r + \frac{\underline{C}}{r}\\
      &\ge& \mathbb{E}\left[\int_0^{\tau^{ij}} e^{-rt} \left(-L\left(X^{0,i,\lambda}_t, \left(\lambda\left(X^{0,i,\lambda}_t, j\right)\right)_{j \in \mathcal{V}\left(X^{0,i,\lambda}_t\right)}\right) + \underline{C}\right) dt + e^{-r\tau^{ij}}\left(u_j^r + \frac{\underline{C}}{r} \right) \right] \\
    & \ge& \mathbb{E}\left[\int_0^{\tau^{ij}} e^{-rt}dt \right]\left( \inf_k -L\left(k, \left(\lambda(k,j)\right)_{j \in \mathcal{V}(k)} \right) + \underline{C}\right) + \mathbb{E}\left[e^{-r\tau^{ij}}\right] \left(u_j^r + \frac{\underline{C}}{r}\right) \\
    & \ge& \mathbb{E}\left[\tau^{ij}\right]\left(\inf_k -L\left(k, \left(\lambda(k,j)\right)_{j \in \mathcal{V}(k)}\right) + \underline{C} \right)  + u_j^r + \frac{\underline{C}}{r}.\\
\end{eqnarray*}
Therefore,
$$u_j^r - u_i^r \le -\mathbb{E}\left[\tau^{ij}\right]\left(\inf_k -L\left(k, \left(\lambda(k,j)\right)_{j \in \mathcal{V}(k)}\right) + \underline{C}\right).$$
Therefore $r \mapsto u_j^r - u_i^r$ is bounded from above. Reverting the role of $i$ and $j$ we obtain the boundedness from below, hence the result.
\end{enumerate}
\end{proof}

We now state two lemmas that will also be useful to study the limit case $r \to 0$.\\

\begin{lemma}
\label{comparevaluefunc}
Let $\varepsilon>0$. Let $(v_i)_{i \in \mathcal{I}}$ and $(w_i)_{i \in \mathcal{I}}$ be such that
\begin{align*}
     -\varepsilon v_i + H\left(i, \left(v_j - v_i\right)_{j \in \mathcal{V}(i)}\right) \ge -\varepsilon w_i + H\left(i, \left(w_j - w_i\right)_{j \in \mathcal{V}(i)}\right), \quad \forall i \in \mathcal{I}.
\end{align*}
Then $\forall i \in \mathcal I, v_i \le w_i$.\\
\end{lemma}
\begin{proof}
Let us consider $(z_i)_{i \in \mathcal I} = (v_i - w_i)_{i \in \mathcal I}$. Let us choose $i^* \in \mathcal{I}$ such that $z_{i^*} = \max_{i \in \mathcal{I}} z_i$.\\

By definition of $i^*$, we know that $\forall j \in \mathcal{V}(i^*), v_{i^*} - w_{i^*} \ge v_{j} - w_{j}$. So, $\forall j \in \mathcal{V}(i^*), v_{j} - v_{i^*} \le w_{j} - w_{i^*}$, and therefore by \ref{Hnd}
$$H\left(i^*, \left(v_j - v_{i^*}\right)_{j \in \mathcal{V}({i^*})}\right) \le H\left(i^*, \left(w_j - w_{i^*}\right)_{j \in \mathcal{V}(i^*)}\right).$$

By definition of $(w_i)_{i \in \mathcal{I}}$ and $(v_i)_{i \in \mathcal{I}}$, we have therefore $ \varepsilon (v_{i^*} - w_{i^*}) \le  0$.\\

We conclude that $\forall i \in \mathcal I, v_i  - w_i \le v_{i^*} - w_{i^*} \le 0.$\\
\end{proof}

\begin{lemma}
\label{gammaunique}
Let $\eta, \mu \in \mathbb R$. Let $(v_i)_{i \in \mathcal{I}}$ and $(w_i)_{i \in \mathcal{I}}$ be such that
\begin{align*}
    & -\eta + H\left(i, \left(v_j - v_i\right)_{j \in \mathcal{V}(i)}\right) = 0,\quad  \forall i \in \mathcal{I}, \\
    & -\mu + H\left(i, \left(w_j - w_i\right)_{j \in \mathcal{V}(i)}\right) = 0,\quad  \forall i \in \mathcal{I}.
\end{align*}
Then $\eta = \mu$.
\end{lemma}

\begin{proof}
If $\eta > \mu$, then let us consider
$$C = \sup_{i \in \mathcal I} (w_i - v_i) + 1 \quad \text{and} \quad \varepsilon = \frac{\eta - \mu}{\sup_{i \in \mathcal I} (w_i - v_i) - \inf_{i \in \mathcal I} (w_i - v_i) + 1}.$$
From these definitions, we have
$$\forall i \in \mathcal{I}, \quad v_i + C > w_i \quad \text{and} \quad 0 \le \varepsilon (v_i - w_i + C) \le \eta - \mu.$$

We obtain
$$\varepsilon (v_i - w_i + C) \le H\left(i, \left(v_j - v_i\right)_{j \in \mathcal{V}(i)}\right) - H\left(i, \left(w_j - w_i\right)_{j \in \mathcal{V}(i)}\right),$$
and therefore
$$-\varepsilon w_i +  H\left(i, \left(w_j - w_i\right)_{j \in \mathcal{V}(i)}\right) \le - \varepsilon (v_i + C) + H\left(i, \left(\left(v_j + C\right) - \left(v_i+C\right)\right)_{j \in \mathcal{V}(i)}\right).$$

From Lemma \ref{comparevaluefunc} it follows that $\forall i \in \mathcal I, v_i + C \le w_i$, in contradiction with the definition of $C$. Therefore $\eta \le \mu$, and by reverting the role of $\eta$ and $\mu$ we obtain $\eta = \mu$.\\
\end{proof}

We can now prove the main result on the convergence of the stationary problem towards the ergodic one.\\

\begin{proposition}
\label{gammav}
We have:\\
\begin{itemize}
\item $\exists \gamma \in \mathbb R, \forall i \in \mathcal{I}$, $\lim_{r \to 0} ru^r_i = \gamma$.\\
\item There exists a sequence $(r_n)_{n \in \mathbb N}$ converging towards $0$ such that $\forall i \in \mathcal{I}, \left(u_i^{r_n} - u_1^{r_n}\right)_{n \in \mathbb N}$ is convergent. \\
\item For all $i \in \mathcal I$, if $\xi_i = \lim_{n \to +\infty} u_i^{r_n} - u_1^{r_n}$, then we have
\begin{equation}
\label{ergodic}
- \gamma + H\left(i, \left(\xi_j - \xi_i\right)_{j \in \mathcal{V}(i)}\right) = 0.
\end{equation}
\end{itemize}
\end{proposition}

\begin{proof}
From the boundedness results of Lemma \ref{lemmaboundedr}, we can consider a sequence $(r_n)_{n \in \mathbb N}$ converging towards $0$, such that, for all $i \in \mathcal{I}$, the sequences $(r_n u^{r_n}_i)_{n \in \mathbb N}$ and $(u_i^{r_n} - u_1^{r_n})_{n \in \mathbb N}$ are convergent, and we denote by $\gamma_i$ and $\xi_i$ the respective limits. We have
$$0 = \lim_{n\to +\infty}r_n(u_i^{r_n} - u_1^{r_n}) =  \lim_{n\to +\infty}r_n u_i^{r_n} - \lim_{n\to +\infty} r_n u_1^{r_n} = \gamma_i - \gamma_1.$$ Therefore, $\forall i \in \mathcal I, \gamma_i = \gamma_1$, and we denote by $\gamma$ this common limit.\\

Using Eq. \eqref{HJstat}, we have
$$ - r_n u^{r_n}_i + H\left(i, \left(u_j^{r_n} - u_i^{r_n}\right)_{j \in \mathcal{V}(i)}\right) = 0.$$
Passing to the limit when $n \to +\infty$, we obtain
\begin{equation*}
- \gamma + H\left(i, \left(\xi_j - \xi_i\right)_{j \in \mathcal{V}(i)}\right) = 0.
\end{equation*}

In order to complete the proof of the above proposition, we need to prove that $\gamma$ is independent of the choice of the sequence $(r_n)_{n \in \mathbb N}$. But this is a straightforward consequence of Lemma \ref{gammaunique}.
\end{proof}

Regarding the limits $(\xi_i)_{i \in \mathcal{I}}$, they cannot be characterized by Eq. \eqref{ergodic} because of the translation invariance property of the equation. However, when the Hamiltonian functions are increasing with respect to each coordinate (and not only non-decreasing), we have the following proposition:

\begin{proposition}
Assume that $\forall i \in \mathcal I, H(i, \cdot)$ is increasing with respect to each coordinate.\\

Let $(v_i)_{i \in \mathcal{I}}$ and $(w_i)_{i \in \mathcal{I}}$ be such that
\begin{align*}
    & -\gamma + H\left(i, \left(v_j - v_i\right)_{j \in \mathcal{V}(i)}\right) = 0, \quad \forall i \in \mathcal{I}, \\
    & -\gamma + H\left(i, \left(w_j - w_i\right)_{j \in \mathcal{V}(i)}\right) = 0, \quad \forall i \in \mathcal{I}.
\end{align*}
Then $\exists C, \forall i \in \mathcal{I}, w_i = v_i + C$, i.e. uniqueness is true up to a constant.
\end{proposition}
\begin{proof}
Let us consider $C = \sup_{i \in \mathcal{I}} w_i - v_i$.\\

By contradiction, if there exists $j \in \mathcal{I}$ such that $v_j + C > w_j$, then because $\mathcal G$ is connected, we can find $i^* \in \mathcal{I}$ such that $v_{i^*} + C = w_{i^*}$ and such that there exists $j^* \in \mathcal{V}(i^*)$ satisfying $v_{j^*} + C > w_{j^*}$.\\

Then, using the strict monotonicity of the Hamiltonian functions, we have the strict inequality $H\left(i^*, \left(\left(v_{j} + C\right) - \left(v_{i^*} + C\right)\right)_{j \in \mathcal{V}(i^*)}\right) > H\left(i, \left(w_{j} - w_{i^*}\right)_{j \in \mathcal{V}(i^*)}\right)$, in contradiction with the definition of $(v_i)_{i \in \mathcal{I}}$ and $(w_i)_{i \in \mathcal{I}}$.\\

Therefore $\forall i \in \mathcal{I}, w_i = v_i + C$.\\
\end{proof}

\begin{remark}
\label{mono}The strict monotonicity of the Hamiltonian functions depends on properties of the cost functions $(L(i,\cdot))_{i\in \mathcal I}$. In some sense, it means that there is no incentive to choose intensities equal to $0$, i.e. no incentive to pay a cost in order to cut existing edges between nodes.
\end{remark}

\section{Asymptotic analysis of the initial finite-horizon control problem in the non-discounted case}
\label{lt}

We now come to the asymptotic analysis of the initial finite-horizon control problem when $r=0$.\\

We have seen in Sections \ref{notation} and \ref{fh} that solving Problem \eqref{controlpbm} boils down to solving Eq. \eqref{mainEQH} with terminal condition \eqref{terminal condition}. Reversing the time over $(-\infty,T]$ by posing $\forall i \in \mathcal I, U_i : t \in \mathbb{R}_+^* \mapsto u^{T,0}_i(T-t)$, this equation, in the case $r=0$, becomes
\begin{equation}
\label{U}
-\frac{d}{dt} U_i(t) + H\left(i, \left(U_j(t) - U_i(t)\right)_{j \in \mathcal{V}(i)}\right) = 0, \quad \forall (i,t) \in \mathcal{I}\times\mathbb{R}_+^*, \quad \text{with } \forall i \in \mathcal{I}, \quad U_i(0) = g(i).
\end{equation}

To carry out the asymptotic analysis, i.e. in order to study the behavior of $(U_i(T))_{i \in \mathcal I}$ as $T \to +\infty$, we assume until the end of this paper that for all $i \in \mathcal{I}$, the function $H(i, \cdot) : p \in \mathbb{R}^{|\mathcal{V}(i)|} \mapsto H(i, p)$ is increasing with respect to each coordinate (see Remark \ref{mono} for a discussion on this strict monotonicity assumption).\\

We introduce the function $\hat{v} : (i,t) \in \mathcal{I} \times [0, +\infty) \mapsto U_i(t) - \gamma t$ where $\gamma$ is given by Proposition \ref{gammav}. Our goal is to study the asymptotic behavior of $\hat{v}$. Let us start with a lemma.

\begin{lemma}
\label{vbounded}
$\hat{v}$ is bounded.
\end{lemma}

\begin{proof}

Let us define for $C \in \mathbb R$ to be chosen, the function $$w^C : (i,t) \in \mathcal{I} \times [0,+\infty) \mapsto w^C_i(t) = \gamma t + \xi_i + C.$$ We have
\begin{align*}
-\frac{d}{dt}{w^C_i}(t) + H\left(i, \left(w^C_j(t) - w^C_i(t)\right)_{j \in \mathcal{V}(i)}\right) = - \gamma + H\left(i, \left(\xi_j - \xi_i\right)_{j \in \mathcal{V}(i)}\right) = 0, \quad \forall (i,t) \in \mathcal{I}\times [0,+\infty).
\end{align*}
By choosing $C_1 = \inf_{i \in \mathcal I} (g(i) - \xi_i)$, we have $\forall i \in \mathcal{I}, w^{C_1}_i(0) \le U_i(0)$. By using Proposition \ref{CompPrinc}, we see, after reversing the time, that therefore $\forall (i,t) \in \mathcal I \times [0, +\infty), w^{C_1}_i(t) \le U_i(t)$.\\

By choosing $C_2 = \sup_{i \in \mathcal I} (g(i) - \xi_i)$, we have $\forall i \in \mathcal{I}, U_i(0) \le w^{C_2}_i(0)$ and therefore, by the same reasoning, $\forall (i,t) \in \mathcal I \times [0, +\infty), U_i(t) \le w^{C_2}_i(t)$.\\

We conclude that $\forall (i,t) \in \mathcal I \times [0, +\infty), \xi_i + C_1 \le \hat{v}_i(t) \le \xi_i + C_2$.\\
\end{proof}

Now, for all $(s,y) \in \mathbb{R}_+ \times \mathbb{R}^N$, let us introduce the Hamilton-Jacobi equation

\begin{align}
\label{longtermeq}
\tag{$E_{s,y}$}
-\frac{d}{dt}\hat{y}_i(t) -\gamma + H\left(i, \left(\hat{y}_j(t) - \hat{y}_i(t)\right)_{j \in \mathcal{V}(i)}\right) = 0, \forall (i,t) \in \mathcal{I}\times[s, +\infty), \quad \textrm{with } \hat{y}_i(s) = y_i, \forall i \in \mathcal{I}.
\end{align}

Using the same reasoning as in Theorem \ref{global}, we easily see that for all $(s,y) \in \mathbb{R}_+ \times \mathbb{R}^N$, there exists a unique global solution to \eqref{longtermeq}.\\

The reason for introducing these equations lies in the following straightforward proposition regarding the function $\hat{v}$.\\
\begin{proposition} Let $y = (y_i)_{i \in \mathcal I} = \left(g(i)\right)_{i \in \mathcal I}$. Then $\hat{v}$ is the solution of $(E_{0,y})$.\\
\end{proposition}

Eq. \eqref{longtermeq} satisfies the following comparison principle which is analogous to that of Proposition \ref{CompPrinc}.

\begin{proposition}[Comparison principle]
\label{longtermcomppr1}
Let $s \in \mathbb R_+$. Let $(\underline{y}_i)_{i \in \mathcal{I}}$ and $(\overline{y}_i)_{i \in \mathcal{I}}$ be two continuously differentiable functions on $[s, +\infty)$ such that
\begin{align*}
    & - \frac{d}{dt}\underline{y}_i(t) - \gamma + H\left(i, \left(\underline{y}_j(t) - \underline{y}_i(t)\right)_{j \in \mathcal{V}(i)}\right) \ge 0,\quad \forall (i,t) \in \mathcal{I} \times [s, +\infty),\\
    & - \frac{d}{dt}\overline{y}_i(t) - \gamma + H\left(i, \left(\overline{y}_j(t) - \overline{y}_i(t)\right)_{j \in \mathcal{V}(i)}\right) \le 0,\quad \forall (i,t) \in \mathcal{I} \times [s, +\infty),
\end{align*}
and $\forall i \in \mathcal I, \underline{y}_i(s) \le \overline{y}_i(s)$.\\

Then $\underline{y}_i(t) \le \overline{y}_i(t),  \forall (i,t) \in \mathcal{I} \times [s, +\infty)$.\\
\end{proposition}

Let us show that the strict monotonicity assumption on the Hamiltonian functions induces in fact a strong maximum principle.\\

\begin{proposition}[Strong maximum principle]
\label{longtermcomppr2}
Let $s \in \mathbb R_+$. Let $(\underline{y}_i)_{i \in \mathcal{I}}$ and $(\overline{y}_i)_{i \in \mathcal{I}}$ be two continuously differentiable functions on $[s, +\infty)$ such that
\begin{align*}
    & - \frac{d}{dt}\underline{y}_i(t) - \gamma + H\left(i, \left(\underline{y}_j(t) - \underline{y}_i(t)\right)_{j \in \mathcal{V}(i)}\right) = 0,\quad \forall (i,t) \in \mathcal{I} \times [s, +\infty),\\
    & - \frac{d}{dt}\overline{y}_i(t) - \gamma + H\left(i, \left(\overline{y}_j(t) - \overline{y}_i(t)\right)_{j \in \mathcal{V}(i)}\right) = 0,\quad \forall (i,t) \in \mathcal{I} \times [s, +\infty),
\end{align*}
and $\underline{y}(s) \lneq \overline{y}(s)$, i.e. $\forall j \in \mathcal I, \underline{y}_j(s) \le \overline{y}_j(s)$ and $\exists i \in \mathcal{I}, \underline{y}_i(s) < \overline{y}_i(s)$.\\

Then $\underline{y}_i(t) < \overline{y}_i(t),  \forall (i,t) \in \mathcal{I} \times (s, +\infty)$.\\
\end{proposition}
\begin{proof}

Using the above comparison principle, and reasoning by contradiction, we assume that there exists $(i,\bar{t}) \in \mathcal{I} \times (s, +\infty)$ such that $\underline{y}_i(\bar{t}) = \overline{y}_i(\bar{t})$. In particular, $\bar{t}$ is a maximizer of the function $ t \in (s,+\infty) \mapsto \underline{y}_i(t) - \overline{y}_i(t)$.\\

Therefore, $\frac{d}{dt}\underline{y}_i(\bar t) = \frac{d}{dt}\overline{y}_i(\bar t)$, and we have subsequently that for all $i \in \mathcal I$ such that $\underline{y}_i(\bar{t}) = \overline{y}_i(\bar{t})$,
$$H\left(i, \left(\underline{y}_j(\bar t) - \underline{y}_i(\bar t)\right)_{j \in \mathcal{V}(i)}\right) = H\left(i, \left(\overline{y}_j(\bar t) - \overline{y}_i(\bar t)\right)_{j \in \mathcal{V}(i)}\right).$$

Let us show now by contradiction that $\forall j \in \mathcal I, \underline{y}_j(\bar{t}) = \overline{y}_j(\bar{t})$.\\

If there exists $j \in \mathcal{I}$ such that $\underline{y}_j(\bar t) \neq \overline{y}_j(\bar t)$ (and then $\underline{y}_j(\bar t) < \overline{y}_j(\bar t)$), then because the graph $\mathcal G$ is connected, we can find $i^* \in \mathcal{I}$ such that $\underline{y}_{i^*}(\bar t) = \overline{y}_{i^*}(\bar t)$ and such that there exists $j^* \in \mathcal{V}(i^*)$ satisfying $\underline{y}_{j^*}(\bar t) < \overline{y}_{j^*}(\bar t)$. From the strict monotonicity assumption on $H(i^*, \cdot)$, we obtain
$$H\left(i^*, \left(\underline{y}_{j}(\bar t) - \underline{y}_{i^*}(\bar t)\right)_{j \in \mathcal{V}(i^*)}\right) < H\left(i^*, \left(\overline{y}_j(\bar t) - \overline{y}_{i^*}(\bar t)\right)_{j \in \mathcal{V}(i^*)}\right).$$

This contradicts the above inequality for $i=i^*$. As a consequence, $\forall j \in \mathcal I, \underline{y}_j(\bar{t}) = \overline{y}_j(\bar{t})$.\\

Let us define $F = \left\{t \in (s, +\infty), \forall j \in \mathcal I, \underline{y}_j(t) = \overline{y}_j(t)\right\}$. $F$ is nonempty since $\bar t \in F$. $F$ is also closed so that $\underline{y}(s) \lneq \overline{y}(s)$ implies that $t^* = \inf F = \min F >s$.\\

We know that $\underline{y}$ and $\overline{y}$ are two local solutions of the Cauchy problem $(E_{t^*,\underline{y}(t^*)})$. Because the Hamiltonian functions are locally Lipschitz, we can apply Cauchy-Lipschitz theorem to conclude that $\underline{y}$ and $\overline{y}$ are in fact equal in a neighborhood of $t^*$, which contradicts the definition of $t^*$.\\

We conclude that $\underline{y}_i(t) < \overline{y}_i(t),  \forall (i,t) \in \mathcal{I} \times (s, +\infty)$.\\
\end{proof}

For all $t \in \mathbb R_+$, let us now introduce the operator $S(t) : y \in \mathbb{R}^N \mapsto \hat{y}(t) \in \mathbb{R}^N$, where $\hat{y}$ is the solution of $(E_{0,y})$.\\

\begin{proposition}
\label{propS}
$S$ satisfies the following properties:\\
\begin{itemize}
    \item $\forall t,t' \in \mathbb R_+, S(t)\circ S(t') = S(t + t') =  S(t') \circ S(t)$.\\
    \item $\forall t \in \mathbb R_+,  \forall x,y \in \mathbb{R}^N,  \left\|S(t)(x) - S(t)(y)\right\|_{\infty} \le \left\| x - y \right\|_{\infty}.$
    In particular, $S(t)$ is continuous.\\
\end{itemize}
\end{proposition}

\begin{proof}
The first point, regarding the semi-group structure, is a natural consequence of Theorem \ref{global} (after a time reversion).\\

For the second point, let us introduce
$$ \underline{y}: t \in \mathbb{R}_+ \mapsto S(t)(x) \quad \text{and} \quad \overline{y}: t \in \mathbb{R}_+ \mapsto  S(t)(y) + \left\|x - y \right\|_{\infty} \Vec{1},$$
where $\Vec{1} = (1, \ldots, 1)' \in \mathbb{R}^N$.\\

We have $\underline{y}(0) = x \le y + \left\|x - y \right\|_{\infty} \Vec{1} = \overline{y}(0)$. By using Proposition \ref{longtermcomppr1}, we have that $\forall t \in \mathbb{R}_+, \underline{y}(t) \le \overline{y}(t)$ and then:
$$\forall t \in \mathbb{R}_+,\quad  S(t)(x) \le S(t)(y) + \left\|x - y \right\|_{\infty} \Vec{1}.$$
Reversing the role of $x$ and $y$ we obtain
\begin{align*}
    &\left\|S(t)(x) - S(t)(y)\right\|_{\infty} \le \left\| x - y \right\|_{\infty}.
\end{align*}
\end{proof}

Now, in order to study the asymptotic behavior of $\hat{v}$, let us define the function
\begin{align*}
&q: t \in \mathbb R_+ \mapsto q(t) = \sup_{i \in \mathcal{I}} (\hat{v}_i(t) - \xi_i).
\end{align*}

\begin{lemma}
\label{qni}
$q$ is a nonincreasing function, bounded from below. We denote by $q_\infty = \lim_{t \rightarrow +\infty} q(t)$ its lower bound.\\
\end{lemma}

\begin{proof}
Let $s \in \mathbb R_+$. Let us define $\underline{y} : (i,t) \in \mathcal{I} \times [s, \infty) \mapsto \hat{v}_i(t)$ and $\overline{y} : (i,t) \in \mathcal{I} \times [s, \infty) \mapsto q(s) + \xi_i$.\\

We have $\forall i \in \mathcal I, \underline{y}_i(s) \le \overline{y}_i(s)$ and
$$-\frac{d}{dt}\overline{y}_i(t) - \gamma + H\left(i, \left(\overline{y}_j(t) - \overline{y}_i(t)\right)_{j \in \mathcal{V}(i)}\right) = -\gamma + H\left(i, \left(\xi_j-\xi_i\right)_{j \in \mathcal{V}(i)}\right) = 0, \forall (i,t) \in \mathcal{I}\times[s,+\infty).$$

From Proposition \ref{longtermcomppr1} we conclude that $\forall (i,t) \in \mathcal{I}\times[s,+\infty), \underline{y}_i(t) \le \overline{y}_i(t)$, i.e. $\hat{v}_i(t) \le q(s) + \xi_i$.\\

In particular, we obtain
\begin{align*}
    & q(t) = \sup_{i \in \mathcal{I}} (\hat{v}_i(t) - \xi_i) \le q(s), \quad \forall t \ge s.
\end{align*}

Now, because $\hat{v}$ is bounded, $q$ is also bounded and we know that its lower bound is its limit $q_\infty = \lim_{t \rightarrow +\infty} q(t)$.\\

\end{proof}

We can now state the main mathematical result of this section.\\

\begin{theorem}
The asymptotic behavior of $\hat{v}$ is given by
$$\forall i \in \mathcal I, \lim_{t \to +\infty} \hat{v}_i(t) = \xi_i + q_{\infty}.$$
\end{theorem}

\begin{proof}

From Lemma \ref{vbounded}, we know that there exists a sequence $(t_n)_n$ converging towards $+\infty$ such that $(\hat{v}(t_n))_n$ is convergent. Let us define $\hat{v}_{\infty} = \lim_{n \to +\infty} \hat{v}(t_n) $.\\

Let us consider the set $\mathcal{K} = \left\{ s \in [0,1] \mapsto \hat{v}(t_n+s) | n \in \mathbb{N}\right\}$. Because $\hat{v}$ is bounded and satisfies $(E_{0,y})$ for $y = (y_i)_{i \in \mathcal I} = (g(i))_{i \in \mathcal I}$, we know from Arzelà–Ascoli theorem that $\mathcal{K}$ is relatively compact in $C^0\left([0,1],\mathbb{R}^N\right)$. In other words, there exists a subsequence $\left(t_{\phi(n)}\right)_n$ and a function $z \in C^0\left([0,1],\mathbb R^N\right)$ such that the sequence of functions $\left(s \in [0,1] \mapsto \hat{v}\left(t_{\phi(n)}+s\right)\right)_n$ converges uniformly towards $z$. In particular $z(0) = \hat{v}_{\infty}$.\\

For all $n \in \mathbb{N}$ and for all $i \in \mathcal I$, we have $\hat{v}_i\left(t_{\phi(n)}\right)\le \xi_i + q(t_{\phi(n)})$, hence $z(0) = \hat{v}_{\infty} \le \xi + q_{\infty} \Vec{1}$.\\

Using Proposition \ref{propS}, we have $$\forall t \in [0,1], S(t)(z(0)) = S(t)\left(\lim_{n \to +\infty} \hat{v}\left(t_{\phi(n)}\right)\right) =  \lim_{n \to +\infty} S(t)\left(\hat v\left(t_{\phi(n)}\right)\right) = \lim_{n \to +\infty} \hat{v}\left(t+t_{\phi(n)}\right) = z(t).$$

As a consequence, we have
$$- \frac{d}{dt}z_i(t) - \gamma + H\left(i, \left(z_j(t) - z_i(t)\right)_{j \in \mathcal{V}(i)}\right) = 0,\quad \forall (i,t) \in \mathcal{I} \times [0,1].$$

Now, let us assume that $z(0) = \hat{v}_{\infty} \lneq \xi + q_{\infty} \Vec{1}$. By using Proposition \ref{longtermcomppr2} we obtain that $z(1) < \xi + q_{\infty} \Vec{1}$ and therefore there exists $n \in \mathbb N$ such that $\hat{v}\left(t_{\phi(n)} + 1\right) < \xi +q_\infty \Vec{1}$. However, this implies the inequality  $q\left(t_{\phi(n)} + 1\right) < q_\infty$ which contradicts the result of Lemma \ref{qni}.\\

This means that $\hat{v}_\infty = \xi + q_\infty \Vec{1}$.\\

In other words, for any sequence $(t_n)_n$ converging towards $+\infty$ such that $(\hat{v}(t_n))_n$ is convergent, the limit is $\xi + q_\infty \Vec{1}$. This means that in fact
$$\forall i \in \mathcal I, \lim_{t \to +\infty} \hat{v}_i(t) = \xi_i + q_{\infty}.$$
\end{proof}

\begin{remark}
In the proof of the above results, the convexity of the Hamiltonian functions $(H(i,\cdot))_{i \in \mathcal I}$ does not play any role. The results indeed hold as soon as the Hamiltonian functions are locally Lipschitz and increasing with respect to each coordinate.\\
\end{remark}

The following straightforward corollary states the asymptotic behavior of the value functions and optimal controls associated with the initial finite-horizon control problem when $r=0$.\\

\begin{corollary}
The asymptotic behavior of the value functions associated with Problem \eqref{controlpbm} when $r=0$ is given by
$$\forall i \in \mathcal I, \forall t \in \mathbb{R}_+, u^{T,r}_i(t) = \gamma (T-t) + \xi_i + q_{\infty} + \underset{T\to+\infty}{o}(1).$$
The limit points of the associated optimal controls for all $t \in \mathbb{R}_+$ as $T \to +\infty$ are feedback control functions verifying
$$\forall i \in \mathcal I, \forall j \in \mathcal V(i), \quad \lambda(i,j) \in \underset{\left(\lambda_{ij}\right)_{j \in \mathcal{V}(i)} \in {\mathbb{R}}^{|\mathcal{V}(i)|}_+ }{\textrm{argmax}} \left(\left(\sum_{j \in \mathcal{V}(i)} \lambda_{ij} (\xi_j - \xi_i)\right) - L\left(i, \left(\lambda_{ij}\right)_{j \in \mathcal{V}(i)}\right)\right).$$
\end{corollary}

\begin{remark}
    If the Hamiltonian functions $(H(i, \cdot))_i$ are differentiable (which is guaranteed if $(L(i, \cdot))_i$ are convex functions that are strictly convex on their respective domain) then the above corollary states in particular the convergence, for all $t \in \mathbb{R}_+$, as $T \to +\infty$ of the optimal controls of Theorem~\ref{theocontrol} towards the unique element of the above $\textrm{argmax}$.\\
\end{remark}

\bibliographystyle{plain}

\section*{Compliance with Ethical Standards}

\begin{itemize}
  \item Conflict of Interest: The authors declare that they have no conflict of interest.
  \item Funding: The authors declare that there was no funding associated with this work.
\end{itemize}

\end{document}